%% file: asymp_spectrum.tex
\documentclass{amsart}

\input{praeambel}

\input{makros}
\usepackage[colorlinks,breaklinks]{hyperref}
\usepackage[T1]{fontenc}
\usepackage{pdflscape}
\allowdisplaybreaks

\begin{document}

\title[]{Asymptotic expansion for the eigenvalues of a perturbed anharmonic oscillator}
\author[K.\@ Fedosova]{Ksenia Fedosova}
\address{KF: Albert-Ludwigs-Universit\"at Freiburg, Mathematisches Institut, Ernst-Zermelo-Str. 1, 79104 Freiburg im Breisgau, Germany}
\email{ksenia.fedosova@math.uni-freiburg.de}
\author[M.\@ Nursultanov]{Medet Nursultanov}
\address{MN: Chalmers University of Technology and University of Gothenburg,
	SE-412 96, Gothenburg, Sweden}
\email{medet.nursultanov@gmail.com}
\begin{abstract}
In this article, we study the spectral properties of the perturbation of the generalized anharmonic oscillator. We consider a piecewise H\"older continuous perturbation and investigate how the H\"older constant can affect the eigenvalues. More precisely, we derive several first terms in the asymptotic expansion for the eigenvalues.
\end{abstract}
\maketitle

\section{Introduction}
Spectral properties of the Sturm-Liouville operators have been studied for more than a century due to numerous applications in mathematics, mechanics, physics and other branches of natural sciences. One of the special cases is the anharmonic oscillator (AHO), which is defined as
\begin{equation*}
H\equiv-\frac{d^2}{dx^2}+q(x), 
\quad q(x)=\sum_{j=1}^{m}a_jx^{2j},
\quad x \in \R.
\end{equation*}
where $a_m>0$. To mention few of many applications, this class of operators provides an equivalent approach to $\lambda \phi^4$-field theory \cite{bender1969anharmonic}, AHOs are used in vibrational spectroscopy as a model for diatomic molecules \cite{sathyanarayana2015vibrational}, and besides, AHOs describe the thermal expansion of crystals \cite{crystals}.

The most famous type of AHO is the harmonic oscillator (HO) that is $q(x)=x^2$. Due to the importance in physics and being a simple and an elegant model, HO is well understood. In particular, its eigenvalues equal $\lambda_n=2n-1$, $n \in \N$ and the corresponding normalized eigenfunctions are explicitly expressed in terms of the Chebyshev-Hermite polynomials; see for instance \cite{LS}. Moreover, there are a lot of publications about perturbed HO, $q(x)=x^2+V(x)$; we mention works concerning the spectral asymptotics\cite{PS,A,Gurarie,Kostenko, KKP,Sakhnovich}.

A HO perturbed by a smooth compactly supported perturbation $V(x)$ is considered in \cite{PS}, where for the corresponding eigenvalues, a complete asymptotic expansion and trace formulas are obtained in terms of the heat invariants. In \cite{KKP}, the authors study a perturbation of HO by a bounded complex function with bounded derivative and bounded indefinite integral. As a main result they obtain the asymptotics for the eigenvalues. The same question, but for a non-smooth perturbation is considered in \cite{A}. For a real-valued measurable perturbation $V(x)$ with certain decay, the asymptotic formula for eigenvalues is established.

In contrast to the HO, the model of AHO, in general settings, cannot be solved analytically, and thus one has to resort to approximation methods for its solution. Several approaches have been used for the numerical evaluation of the eigenvalue problem, see \cite{GSS} and references therein. 

The spectral properties of AHO were studied in \cite{Ruzhansky,HelfferRobert,Camus,Fucci,CamusRoutenberg}. In \cite{HelfferRobert}, the authors studied, in particular, the operators of the form $-d^{2k}/dx^{2k}+x^{2l}+p(x)$, for $k$, $l\in \mathbb{N}$ and with $p(x)$ being a polynomial of degree less than $2l$. They established the asymptotic formula which describes the behavior of the eigenvalues. In higher dimension, the Laplace operator, perturbed by a smooth radially symmetric polynomial potential on unbounded domain, is considered in \cite{Fucci}. The author obtains the asymptotic expansion of the heat kernel trace. In \cite{MA}, the authors obtain the eigenvalue asymptotic of $H$ in $L^2[0,\infty)$ with the Dirichlet boundary condition and $q(x)=x^{\alpha}+V(x)$, where $\alpha>0$ and $V(x)$ is a real-valued, compactly supported, twice differentiable function on $[0,\infty)$.


In most works concerning the spectral properties of perturbed AHO, the perturbations are smooth. For an actual real-world potential smoothness is not necessarily guaranteed. For this reason, we want to reduce the smoothness and explore how this will affect the eigenvalues, we require $V(x)$ to be only piecewise H\"older-continuous.

\begin{thm}\label{main_th}

Let $H$ be the self-adjoint operator in $L^2(\mathbb{R})$, generated by 
\begin{equation}\label{def_patential1}
-\frac{d^2}{dx^2} + |x|^{\alpha}+V(x),
\end{equation}
where  $\alpha > 0$, and $V(x)$ is a bounded, real-valued, compactly supported, piecewise\footnote{We mean that there is a finite number of pieces such that $V$ is H{\"o}lder continuous on each piece.} H{\"o}lder continuous function with an exponent $\tau>0$. Then the sequence of eigenvalues $\{\lambda_n\}_{n=1}^{\infty}$ of $H$ satisfies the following asymptotic formula\footnote{Note that the constants $C_0$,$C_1$, and $C_2$ are defined slightly differently than in \cite{MA}.}
\begin{align}\label{main_asy}
	\begin{split}
		\lambda_n&=C_1^{-\frac{2\alpha}{\alpha+2}}(2n-1)^{\frac{2\alpha}{\alpha+2}}\\
		&+\frac{2\alpha}{\alpha+2}C_0C_1^{-\frac{\alpha+4}{\alpha+2}}(2n-1)^{-\frac{2}{\alpha+2}}\\
		&+\frac{2\alpha}{\alpha+2}\frac{1}{4\pi}C_1^{-\frac{\alpha+4}{\alpha+2}}(2n-1)^{-\frac{2}{\alpha+2}}\int_{-\infty}^{\infty}V(s)\cos\left(2C_1^{-\frac{\alpha}{\alpha+2}}(2n-1)^{\frac{\alpha}{\alpha+2}}s\right)ds\\
		&+\frac{2\alpha}{\alpha+2}C_2C_1^{-\frac{\alpha+6}{\alpha+2}}(2n-1)^{-\frac{4}{\alpha+2}}+O\left(n^{-1}\right),
	\end{split}
\end{align}
where
\begin{equation*}
	C_1=\frac{4\Gamma\left(\frac{3}{2}\right) \Gamma\left(\frac{1}{\alpha}\right)}{\alpha\pi \Gamma\left(\frac{3}{2}+\frac{1}{\alpha}\right)},
	\quad
	C_0=\frac{1}{\pi}\int_{-\infty}^{\infty}V(s)ds,
	\quad
	C_2=\frac{\alpha-1}{12\pi(2+\alpha)}\cot\left(\frac{\pi}{\alpha}\right)C_1^{-1},
\end{equation*}
    and $\Gamma (\cdot)$ is the gamma function.
\end{thm}

To the best knowledge of the authors, these types of perturbations of anharmonic oscillators have not been previously treated in the physics literature. There has, however, been research on perturbations of an harmonic oscillator by Gaussian noise: \cite{BernardinGoncalvesMilton}, \cite{Gitterman}. Moreover, H\"older continuous potentials in Sturm-Liouville operators (that are not AHO) appear in \cite{Ikebe} and \cite{Schmidt}.

Theorem \ref{main_th} shows that the perturbation, $V(x)$, does not affect the first term. However, it appears in the second term, while the regularity, the parameter $\tau$, affects only the third term. Indeed, in case $V(x)$ being smooth and compactly supported, the third term would decay rapidly. When $V(x)$ is H{\"o}lder continuous with an exponent $\tau>0$, we can say only that the third term is $O(n^{-\frac{\alpha\tau+2}{\alpha+2}})$. In order to demonstrate more explicitly the effect of the smoothness, we construct an example. There we consider the operator $H$ from Theorem \ref{main_th} for $\alpha=2$ and  $V(x)$ being the Weierstrass function defined as in \eqref{exampleV}. Then we find the subsequence of the eigenvalues $\{\lambda_{n_k}\}_{k=1}^{\infty}$ such that
\begin{equation*}
	\lambda_{n_k}=2n_k-1+n_k^{-\frac{1}{2}}\frac{1}{4\sqrt{2}}\int_{-\pi}^{\pi}V(s)ds +n_k^{-\frac{1+\tau}{2}}2^{-\frac{5+3\tau}{2}}+O(n_k^{-1}).
\end{equation*}

To some extent, the proof of Theorem \ref{main_th} generalizes to more general potentials. However, the results are not so  explicit:

\begin{thm}\label{gen_th}
	Assume that $\{a_j\}_{j=1}^{N}$ and $\{\alpha_j\}_{j=1}^{N}$ are sets of real numbers such that $0<\alpha_1<...<\alpha_{N}$ and $a_N>0$. Let $H$ be a self-adjoint operator in $L^2(\mathbb{R})$, generated by the expression
	\begin{equation}\label{def_patential}
	-\frac{d^2}{dx^2} + q(x), \quad q(x) = \sum_{j=1}^N a_j |x|^{\alpha_j}+V(x),
	\end{equation}
	where $V(x)$ is a bounded, real-valued, piecewise H{\"o}lder continuous function with an exponent $\tau>0$, compactly supported in $(-b,b)$ for some $b>0$, which we consider to be fixed. Then the sequence of eigenvalues $\{\lambda_n\}_{n=1}^{\infty}$ of $H$ satisfies
	\begin{align}\label{asy_gen}
		\begin{split}
			\frac{\pi}{4}(2n-1)=&Q(b,\lambda_n)+b\sqrt{\lambda_n-q(b)}-\frac{1}{4\sqrt{\lambda_n}}\int_{-b}^{b}(q(s)-q(b))ds\\
			&-\frac{1}{4\sqrt{\lambda_n}}\int_{-b}^{b}V(s)\cos(2\sqrt{\lambda_n}s)ds+O\left(\lambda_n^{-\frac{\alpha+2}{2\alpha}}\right)+O\left(\lambda_n^{-1}\right).
		\end{split}
	\end{align}
	where $\alpha:=\alpha_N$,
	\begin{equation*}
		Q(x,\lambda):=\int_{x}^{a(\lambda)}\sqrt{\lambda-q(t)}dt,
	\end{equation*}
	and $a(\lambda)$ is the turning point, that is $q(a(\lambda))=\lambda$ for sufficiently large $\lambda>0$.
\end{thm}

We note that, in some cases, it is easy to express the function $Q(b,\lambda)$ in terms $\{\lambda^s\}$, so that \eqref{asy_gen} can be written more explicitly. In the last section, we give some examples, including a quartic AHO; see also Remark \ref{Remark}.

In Corollary \ref{asy_for_DN}, we consider the sequences of the eigenvalues of the operators in $L_2[0,\infty)$ generated by \eqref{def_patential1} and the Dirichlet and Neumann boundary conditions. We derive their asymptotic formulas, which show that they are interlacing at infinity.

\subsection{Strategy and structure of the paper.}The proof of the theorem uses the idea of \cite{MA} and goes as follows:
 for a sufficiently large $b>0$, we construct solutions, $f_+(x, \lambda)$ and $f_-(x, \lambda)$, of $Hy=\lambda y$ in $[0,b]$ and $[-b,0]$, respectively, satisfying the following boundary conditions
 \[
 f_+(0, \lambda) = f_-(0,\lambda) =  \cos \phi(\lambda), \quad 
 \quad f_+'(0,\lambda) = f_-'(0,\lambda)= \sqrt{\lambda-q(b)} \sin\phi(\lambda)
 \]
 for some $\phi(\lambda) \in [0,2\pi)$. Then we construct (with a different method) a solution, $y(x, \lambda)$,  of $Hy=\lambda y$, that is square-integrable on $[b,\infty)$. A square-integrable solution on $(-\infty, -b]$ can be obtained by a flip $x \mapsto -x$, as the potential, $q(x)$, is symmetric outside the support of the perturbation.

Note that $f_+(x, \lambda)$ can be extended to an $L^2$-solution of $H$ on $[0, \infty)$, if and only if the vector $(f_+(b, \lambda),f'_+(b, \lambda))$ is linearly dependent with $(y(b, \lambda),y'(b, \lambda))$. The linear dependency at points $+b$ and $-b$ gives a system of two equations depending on  $\lambda$ and $\phi(\lambda)$. These two equations imply that as the spectral parameter, $\lambda$, goes to infinity, $\phi(\lambda)$ would be forced to tend either to $0$, or to $\frac{\pi}{2}$.

Note that in the case of the symmetric perturbation, $V(x)$, we would obtain either the equality $\phi = 0$ or $\pi / 2$ straight away, that would correspond to the case of either Dirichlet or Neumann boundary conditions at zero, or, that is the same, would force the solution to be even or odd. So, heuristically we can say that the condition of the potential, $q(x)$, being symmetric at infinity turns out to be strong enough in order to get an asymptotical evenness or oddness of the eigenfunction, as the spectral parameter $\lambda$ tends to infinity.

We split these two cases and obtain an equality on $\lambda$, which holds asymptotically as $\lambda$ tends to infinity, would allows us to obtain the asymptotics of the counting function and the asymptotic behavior of eigenvalues.

It turns out that the aforementioned eigenvalues of "almost odd" and "almost even" (that is, corresponding to the cases $\phi$ is approximately $0$ or $\pi / 2$) eigenfunctions are interlacing.

In order to be sure that the asymptotic equalities allow us to take care of all the eigenvalues of the problem, we give a rough estimate on the asymptotic expansion of the counting function of eigenvalues.

\subsection{Acknowledgements} The first author is grateful to the supervisor of her master's thesis, Vladimir Podolskii, as the article was inspired by the aforementioned master's thesis. The second author was partially supported by the Ministry of Education Science of the Republic of Kazakhstan under the grant AP05132071. We would like to thank Julie Rowlett for reading and commenting upon preliminary version of this manuscript. We are also grateful to Grigori Rozenblum for the attention and useful comments and Simone Murro for helpful discussions.

\section{Preliminaries}
In this section, we first prove that $H$, defined as in \eqref{def_patential}, is a self-adjoint operator. Next, we construct the solutions of 
\begin{equation}\label{diff_eq}
-y''(x,\lambda)+\left(\sum_{j=1}^{N}a_j|x|^{\alpha_j} + V(x)\right)y(x,\lambda)=\lambda y(x,\lambda)
\end{equation}

in $[0, b]$ with certain  boundary conditions at $x = 0$, and in $[b, \infty)$ under the condition that the solution is square-integrable. Above, $b > 0$ is  such that the perturbation, $V$, is compactly supported in $(-b,b)$. Finally, we study the asymptotic behavior of these solutions at a point $b$ as the spectral parameter, $\lambda$, tends to infinity.

Let $\{a_j\}_{j=1}^{N}$, $\{\alpha_j\}_{j=1}^{N}$ be sets of positive numbers and $V(x)$ be a bounded, real valued function supported inside $(-b,b)$. We also require $V(x)$ to be a piecewise H{\"o}lder continuous function with an exponent $\tau>0$, that is there exist $C>0$ and $-\infty=x_0<x_1...<x_m=+\infty$, $m\in\mathbb{N}$, such that
\begin{equation*}
\left|V(x)-V(y)\right|\leq C|x-y|^{\tau},
\end{equation*}
for all $x,y\in (x_j,x_{j+1})$ and $j=0,...,m$. Consider the differential expression \eqref{def_patential}. Note that the operator of multiplication by $V(x)$ is symmetric and bounded, whilst the operator associated with the non perturbed AHO in $C_{c}^{\infty}(\mathbb{R})$ is essentially self-adjoint; see \cite{LS}. The Rellich-Kato theorem \cite[Theorem 4.4]{Kato} implies that the operator defined by \eqref{def_patential} in $C_{c}^{\infty}(\mathbb{R})$ is essentially self-adjoint as well and hence has a unique self-adjoint extension, which we denote by $H$.

\subsection{Construction of solutions outside the support of the perturbation}
We start by considering \eqref{diff_eq} in $[b,\infty)$. Note that in this interval the perturbation,  $V(x)$, equals zero and hence \eqref{diff_eq} becomes
\begin{equation}\label{eq_on_b_infty}
     -y''(x,\lambda)+\sum_{j=1}^{N}a_j|x|^{\alpha_j}y(x,\lambda)=\lambda y(x,\lambda).
\end{equation}
Its solutions have already been studied in \cite{MA}, so that we recall  some of their results.

Let $y(x,\lambda)$ satisfy \eqref{eq_on_b_infty} and $Q(x,\lambda)$ be the function defined as in Theorem \ref{gen_th}. Define the function
\begin{equation*}
	\eta(x,\lambda):=|\lambda-q(x)|^{1/4}y(x,\lambda).
\end{equation*}
 According to \cite[Lemma 2]{MA}, it follows that
\begin{equation}\label{eta}
\begin{gathered}
\eta(b, \lambda) = d_1(\lambda) \sin\left(Q(b, \lambda) + \frac{\pi}{4}\right) - d_2(\lambda) \cos\left(Q(b, \lambda)+\frac{\pi}{4}\right) 
+ O\left(\lambda^{-\frac{\alpha+2}{\alpha}}\right),\\
\end{gathered}
\end{equation}
\begin{align}\label{eta_prime}
\begin{split}
\eta'(b, &\lambda) = \left. \frac{\partial}{\partial x} \eta(x, \lambda) \right|_{x = b}\\
& = - \sqrt{\mu(\lambda)} \left[ d_1(\lambda)\cos\left(Q(b, \lambda) + \frac{\pi}{4}\right) +
d_2(\lambda) \sin\left(Q(b,\lambda)+\frac{\pi}{4}\right)\right] +O\left(\lambda^{-\frac{\alpha+4}{2\alpha}}\right)
\end{split}
\end{align}
as $\lambda\rightarrow\infty$, where $\mu(\lambda):=\lambda-q(b)$ and $d_1(\lambda)$, $d_2(\lambda)$ are  given explicitly in terms of Airy functions. Moreover, by \cite[(36)]{MA}, as $\lambda\rightarrow\infty$,
\begin{equation}\label{asympt_for_d}
d_1(\lambda)=1+O(\lambda^{-\frac{\alpha+2}{2\alpha}}),
\qquad
d_2(\lambda)=O(\lambda^{-\frac{\alpha+2}{2\alpha}}).
\end{equation}

\subsection{Construction of solutions inside the support of the perturbation}
Next we study the differential equation \eqref{diff_eq} in $[0,b]$ with the  boundary conditions
\begin{equation*}
y(0)=c_1=c_1(\lambda),\qquad
y'(0)=c_2=c_2(\lambda).
\end{equation*}
Note its solutions are in one-to-one correspondence with the solutions of the following integral equation:\footnote{We have chosen the notation in such a way that $f_+$ corresponds to a segment of a positive half-line, $[0, b]$,  and $f_-$ corresponds to a segment of a negative half-line,  $[-b,0]$.}
\begin{align}\label{eleven_a}
\begin{split}
f_{+}(x,\lambda)=&
c_1\cos\sqrt{\mu(\lambda)}x+c_2\frac{\sin\sqrt{\mu(\lambda)}x}{\sqrt{\mu(\lambda)}}\\
&+\frac{1}{\sqrt{\mu(\lambda)}}\int_{0}^{x}\sin(\sqrt{\mu(\lambda)}(x-s))[q(s)-q(b)]f_{+}(s,\lambda)ds.
\end{split}
\end{align}
This is a Volterra equation, and thus has a unique solution, $f_+(x,\lambda)$; see \cite[page 398]{E}. Differentiating it, we obtain
\begin{align}\label{eleven_a_der}
\begin{split}
f_{+}'(x,\lambda)=&
-c_1\sqrt{\mu(\lambda)}\sin\sqrt{\mu(\lambda)}x+c_2\cos\sqrt{\mu(\lambda)}x\\
&+\int_{0}^{x}\cos(\sqrt{\mu(\lambda)}(x-s))[q(s)-q(b)]f_{+}(s,\lambda)ds.
\end{split}
\end{align}

We write \eqref{eleven_a} and \eqref{eleven_a_der} in the following way
\begin{align}\label{eq15}
\begin{split}
f_{+}(b,\lambda)&=\cos(\sqrt{\mu(\lambda)}b)\left[c_1-\frac{1}{\sqrt{\mu(\lambda)}}\int_{0}^{b}\sin(\sqrt{\mu(\lambda)}s)[q(s)-q(b)]f_{+}(s,\lambda)ds\right] \\
&+\sin(\sqrt{\mu(\lambda)}b)\left[\frac{c_2}{\sqrt{\mu(\lambda)}}+\frac{1}{\sqrt{\mu(\lambda)}}\int_{0}^{b}\cos(\sqrt{\mu(\lambda)}s)[q(s)-q(b)]f_{+}(s,\lambda)ds\right].
\end{split}
\end{align}
and
\begin{align}\label{eq16}
\begin{split}
f_{+}'(b,\lambda) & =\cos(\sqrt{\mu(\lambda)}b)\left[c_2+\int_{0}^{b}\cos(\sqrt{\mu(\lambda)}s)[q(s)-q(b)]f_{+}(s,\lambda)ds\right] \\
& +\sin(\sqrt{\mu(\lambda)}b)\left[-c_1\sqrt{\mu(\lambda)}+\int_{0}^{b}\sin(\sqrt{\mu(\lambda)}s)[q(s)-q(b)]f_{+}(s,\lambda)ds \right].
\end{split}
\end{align}

Define
\begin{equation*}
k^{+}_1(\lambda):=\int_{0}^{b}\cos(\sqrt{\mu(\lambda)}s)[q(s)-q(b)]f_{+}(s,\lambda)ds,
\end{equation*}
\begin{equation*}
k^{+}_2(\lambda):=\int_{0}^{b}\sin(\sqrt{\mu(\lambda)}s)[q(s)-q(b)]f_{+}(s,\lambda)ds.
\end{equation*}
Hence \eqref{eq15} and \eqref{eq16} can be rewritten as
\begin{equation}\label{f_plus_f_prime_plus}
	\begin{gathered}
f_{+}(b,\lambda)=\cos(\sqrt{\mu(\lambda)}b)\left(c_1-\frac{k^{+}_2(\lambda)}{\sqrt{\mu(\lambda)}}\right)+\sin(\sqrt{\mu(\lambda)}b)\left(\frac{c_2}{\sqrt{\mu(\lambda)}}+\frac{k^+_1(\lambda)}{\sqrt{\mu(\lambda)}}\right),\\
f_{+}'(b,\lambda)=\cos(\sqrt{\mu(\lambda)}b)\left(c_2+k^{+}_1(\lambda)\right)+\sin(\sqrt{\mu(\lambda)}b)\left(-c_1\sqrt{\mu(\lambda)}+k^+_2(\lambda)\right).
\end{gathered}
\end{equation}

Next, we consider \eqref{diff_eq} in $[-b,0]$. This is equivalent to considering a "reflected" equation
\begin{equation}\label{eq_on_b_0}
-y''(x)+q(-x)y(x)=\lambda y(x),\qquad x\in[0,b],\\
\end{equation}
with boundary conditions
\begin{equation*}
y(0)=c_1=c_1(\lambda),\qquad
y'(0)=-c_2=c_2(\lambda).
\end{equation*}

Similarly to \eqref{eleven_a}, we construct the solution of
\begin{align*}
f_{-}(x,\lambda)&=
c_1\cos\sqrt{\mu(\lambda)}x-c_2\frac{\sin\sqrt{\mu(\lambda)}x}{\sqrt{\mu(\lambda)}}\\
&+\frac{1}{\sqrt{\mu(\lambda)}}\int_{0}^{x}\sin(\sqrt{\mu(\lambda)}(x-s))[q(-s)-q(-b)]f_{-}(s,\lambda)ds
\end{align*}
and derive
\begin{align*}
	f_{-}(b,\lambda)&=\cos(\sqrt{\mu(\lambda)}b)\left(c_1-\frac{k^{-}_2(\lambda)}{\sqrt{\mu(\lambda)}}\right)+\sin(\sqrt{\mu(\lambda)}b)\left(-\frac{c_2}{\sqrt{\mu(\lambda)}}+\frac{k^-_1(\lambda)}{\sqrt{\mu(\lambda)}}\right), \\	f_{-}'(b,\lambda)&=\cos(\sqrt{\mu(\lambda)}b)\left(-c_2+k^{-}_1(\lambda)\right)+\sin(\sqrt{\mu(\lambda)}b)\left(-c_1\sqrt{\mu(\lambda)}+k^-_2(\lambda)\right),
\end{align*}
where
\begin{equation*}
	k^{-}_1(\lambda):=\int_{0}^{b}\cos(\sqrt{\mu(\lambda)}s)[q(-s)-q(-b)]f_{-}(s,\lambda)ds,
\end{equation*}
\begin{equation*}
	k^{-}_2(\lambda):=\int_{0}^{b}\sin(\sqrt{\mu(\lambda)}s)[q(-s)-q(-b)]f_{-}(s,\lambda)ds.
\end{equation*}

In the following two lemmas we investigate the asymptotic behavior of the functions $f_{\pm}(s,\lambda)$, $k_1^{\pm}(\lambda)$ and $k_2^{\pm}(\lambda)$  as $\lambda\rightarrow+\infty$.

\begin{lemma}\label{f+est}
	The solutions $f_{\pm}(s,\lambda)$ satisfy, as $\lambda\rightarrow+\infty$,
	\begin{equation*}
		f_{\pm}(x,\lambda) = c_1(\cos\sqrt{\mu(\lambda)}x)\pm c_2\frac{\sin(\sqrt{\mu(\lambda)}x)}{\sqrt{\mu(\lambda)}} + c_1O\left(\lambda^{-\frac{1}{2}}\right) + c_2  O\left(\lambda^{-1}\right).
	\end{equation*}
\end{lemma}
\begin{proof}
	Denote by $H_{\lambda}$ the operator
	$$
	H_{\lambda} : u \mapsto \mu(\lambda)^{-1/2} \int_0^x [q(s) - q(b)] \sin(\sqrt{\mu(\lambda)}(x-s)) u(s) ds
	$$
	in $C[0,b]$. Then
	\begin{equation}\label{f+viaj}
	f_{+}(x,\lambda) = c_1\cos(\sqrt{\mu(\lambda)}x)+c_2\frac{\sin(\sqrt{\mu(\lambda)}x)}{\sqrt{\mu(\lambda)}} + (H_{\lambda} f_{+})(x).
	\end{equation}
	By the triangle inequality
	\begin{equation}
	\|f_{+}(\cdot ,\lambda)\|_{\infty}\leq  |c_1| + \frac{|c_2|}{\sqrt{\mu(\lambda)}}  + \|H_{\lambda}\|_{op}\cdot\|f_{+}(\cdot, \lambda)\|_{\infty},
	\end{equation}
	where $\|\cdot\|_{\infty}$ is the uniform norm and $\|\cdot\|_{op}$ is the operator norm. By \cite[ (30)]{MA}, $\|H_{\lambda}\|_{op} = O(\lambda^{-\frac{1}{2}})$ as $\lambda \to \infty$, hence for sufficiently large $\lambda$, $1- \| H_ \lambda\|_{op} > 0$ and therefore
	\begin{equation}\label{f+est_1}
	\|f_{+}(\cdot ,\lambda)\|_{\infty} \leq  \frac{|c_1| + \frac{|c_2|}{\sqrt{\mu(\lambda)}} }{1- \| H_ \lambda\|_{op}}.
	\end{equation}	
	 Moreover, \eqref{f+est_1} implies that
	\begin{equation*}
		\| f_+(\cdot, \lambda)\|_{\infty}=c_1 O(1) + \frac{c_2}{\sqrt{\mu(\lambda)}}  O(1), \qquad \lambda \to \infty,
	\end{equation*}
	which together with \eqref{f+viaj} gives
	\begin{equation*}
	\begin{gathered}
	\left\| f_{+}(\cdot,\lambda) - c_1\cos(\sqrt{\mu(\lambda)}(\cdot) ) - c_2\frac{\sin(\sqrt{\mu(\lambda)} (\cdot) )}{\sqrt{\mu(\lambda)}} \right\|_\infty  \le \| H_{\lambda}\|_{op} \cdot || f_+ (\cdot , \lambda) ||_\infty\\
    = c_1  O\left(\lambda^{-\frac{1}{2}}\right) + c_2 O\left(\lambda^{-1}\right).
	\end{gathered}
    \end{equation*}
\end{proof}
Next we derive asymptotic formulas for the functions $k_1^{\pm}(\lambda)$ and $k_2^{\pm}(\lambda)$.
\begin{lemma}\label{k_pus_k}
	The functions $k_1^{\pm}(\lambda)$ and $k_2^{\pm}(\lambda)$ satisfy the following asymptotic formulas, as $\lambda\rightarrow \infty$,
	\begin{equation}\label{k_pus_k_alt_1}
	k_1^{\pm}(\lambda)=\frac{c_1}{2}\int_{0}^{b}\left(q(\pm s)-q(\pm b)\right)ds+c_1O\left(\lambda^{-\frac{\tau}{2}}\right)+\frac{c_2}{\sqrt{\mu(\lambda)}}O\left(\lambda^{-\frac{\tau}{2}}\right),
	\end{equation}
	\begin{equation}\label{k_pus_k_alt_2}
	k_2^{\pm}(\lambda)=\frac{\pm c_2}{2\sqrt{\mu(\lambda)}}\int_{0}^{b}\left(q(\pm s)-q(\pm b)\right)ds+c_1O\left(\lambda^{-\frac{\tau}{2}}\right)+\frac{c_2}{\sqrt{\mu(\lambda)}}O\left(\lambda^{-\frac{\tau}{2}}\right).
	\end{equation}
\end{lemma}
\begin{proof}
	Lemma \ref{f+est} implies that
	\begin{align}\label{k1qwseqwe}
		\begin{split}
		k_1^{+}(\lambda&)  =  \frac{c_1}{2}\int_{0}^{b}[q( s)-q( b)]ds
		+\frac{c_1}{2}\int_{0}^{b}[q( s)-q( b)]\cos(2\sqrt{\mu(\lambda)}s)ds \\
	& +\frac{c_2}{2\sqrt{\mu(\lambda)}}\int_{0}^{b}[q( s)-q( b)]\sin(2\sqrt{\mu(\lambda)}s)ds  +c_1O\left(\lambda^{-\frac{1}{2}}\right)+c_2O\left(\lambda^{-1}\right).
	\end{split}
	\end{align}
	Since $q(x)$ is a piecewise H{\"o}lder continuous function with an exponent $\tau$, by \cite[page 92]{SR}, we obtain
	\begin{align}\label{asurdhzukzhkujh}
	\begin{split}
		\int_{0}^{b}[q( s)-q( b)]\cos(2\sqrt{\mu(\lambda)}s)ds&=O\left(\lambda^{-\frac{\tau}{2}}\right),\\
		\int_{0}^{b}[q( s)-q( b)]\sin(2\sqrt{\mu(\lambda)}s)ds&=O\left(\lambda^{-\frac{\tau}{2}}\right)
	\end{split}
	\end{align}
	as $\lambda\rightarrow\infty$. Combining \eqref{asurdhzukzhkujh} with  \eqref{k1qwseqwe}, we derive \eqref{k_pus_k_alt_1}. Similarly, one  proves \eqref{k_pus_k_alt_2}.
\end{proof}
We also will need sharper asymptotics for $k_1^++k_1^-$ and $k_2^+-k_2^-$. By \eqref{k1qwseqwe}, we derive
\begin{align}\label{K1minusK1}
	\begin{split}
		k_1^+(\lambda)+k_1^-&(\lambda)=\frac{c_1}{2}\int_{-b}^{b}(q(s)-q(b))ds+\frac{c_1}{2}\int_{-b}^{b}(q(s)-q(b))\cos(2\sqrt{\mu(\lambda)}s)ds\\
		&+\frac{c_2}{2\sqrt{\mu(\lambda)}}\int_{-b}^{b}(q(s)-q(b)\sin(2\sqrt{\mu(\lambda)}s)ds+c_1O\left(\lambda^{-\frac{1}{2}}\right)+c_2O\left(\lambda^{-1}\right).
	\end{split}
\end{align}
Similarly, we obtain
\begin{align}\label{K2minusK2}
	\begin{split}
		k_2^+(\lambda)-k_2^-&(\lambda)=\frac{c_2}{2\sqrt{\mu(\lambda)}}\int_{-b}^{b}(q(s)-q(b))ds+\frac{c_1}{2}\int_{-b}^{b}(q(s)-q(b))\sin(2\sqrt{\mu(\lambda)}s)ds\\
		&+\frac{c_2}{2\sqrt{\mu(\lambda)}}\int_{-b}^{b}(q(s)-q(b))\cos(2\sqrt{\mu(\lambda)}s)ds+c_1O\left(\lambda^{-\frac{1}{2}}\right)+c_2O\left(\lambda^{-1}\right).
	\end{split}
\end{align}

\section{Proof of Theorems \ref{main_th} and \ref{gen_th}}
In this section we use the notations introduced in the previous sections. Moreover, we skip writing arguments for the following functions, where $j=1,2$,
\begin{equation*}
	Q:=Q(b,\lambda),
	\quad
	\mu:=\mu(\lambda),
	\quad
	d_j:=d_j(\lambda),
	\quad
	k_j^{\pm}:=k_j^{\pm}(\lambda).
\end{equation*}

We start by proving Theorem \ref{gen_th}.
\begin{proof}
For fixed $c_1, c_2 \in \R$, the solutions $f_+(x,\lambda)$ and $y(x,\lambda)$ are linearly dependent, provided that $\lambda$ is an eigenvalue of $H$ in $L^2(\mathbb{R})$. This gives us the following equation for the eigenvalues
\begin{equation*}
f'_+(b,\lambda)y(b,\lambda)-f_+(b,\lambda)y'(b,\lambda)=0,\\
\end{equation*}
or
\begin{equation*}
f'_+(b,\lambda)\eta(b,\lambda)-f_+(b,\lambda)\eta'(b,\lambda)=(4\mu)^{-1}q'(b)\eta(b,\lambda)f_+(b,\lambda).
\end{equation*}
By inserting \eqref{eta}, \eqref{eta_prime} and \eqref{f_plus_f_prime_plus} into the previous equation and using Lemma \ref{f+est}, we obtain
\begin{align*}
	&\Big[\cos(\sqrt{\mu}b)\left(c_2+k^{+}_1\right)+\sin(\sqrt{\mu}b)\left(-c_1\sqrt{\mu}+k^+_2\right)\Big]\\&
	\times\left[d_1\sin\left(Q+\frac{\pi}{4}\right)-d_2\cos\left(Q+\frac{\pi}{4}\right)+O\left(\lambda^{-\frac{\alpha+2}{\alpha}}\right)\right] \\&
	+\Big[\cos(\sqrt{\mu}b)\left(c_1\sqrt{\mu}-k^+_2\right)+\sin(\sqrt{\mu}b)\left(c_2+k^{+}_1\right)\Big]\\&
	\times\left[d_1\cos\left(Q+\frac{\pi}{4}\right)+d_2\sin\left(Q+\frac{\pi}{4}\right)+O(\lambda^{-\frac{\alpha+2}{\alpha}})\right]\\&
	=\cos(\sqrt{\mu}b)\sin\left(Q+\frac{\pi}{4}\right)\Big(d_1\left(c_2+k_1^+\right)+d_2\left(c_1\sqrt{\mu}-k_2^+\right)\Big) \\&
	+\sin(\sqrt{\mu}b)\cos\left(Q+\frac{\pi}{4}\right)\Big(d_1\left(c_2+k_1^+\right)+d_2\left(c_1\sqrt{\mu}-k_2^+\right)\Big)\\&
	+\cos(\sqrt{\mu}b)\cos\left(Q+\frac{\pi}{4}\right)\Big(d_1\left(c_1\sqrt{\mu}-k_2^+\right)-d_2\left(c_2+k_1^+\right)\Big)\\&
	+\sin(\sqrt{\mu}b)\sin\left(Q+\frac{\pi}{4}\right)\Big(-d_1\left(c_1\sqrt{\mu}-k_2^+\right)+d_2\left(c_2+k_1^+\right)\Big)\\&
	+c_1\sqrt{\mu}O\left(\lambda^{-\frac{\alpha+2}{\alpha}}\right)+c_2O\left(\lambda^{-\frac{\alpha+2}{\alpha}}\right)\\&
		=c_1O\left(\lambda^{-1}\right)+c_2O\left(\lambda^{-\frac{3}{2}}\right).
\end{align*}
Hence
\begin{align}\label{eq_plus}
	\begin{split}
&\sin\left(Q+\frac{\pi}{4}+\sqrt{\mu}b\right)\Big(d_1\left(c_2+k_1^+\right)+d_2\left(c_1\sqrt{\mu}-k_2^+\right)\Big)\\&
	+\cos\left(Q+\frac{\pi}{4}+\sqrt{\mu}b\right)\Big(d_1\left(c_1\sqrt{\mu}-k_2^+\right)-d_2\left(c_2+k_1^+\right)\Big)\\&
	=c_1\sqrt{\mu}O\left(\lambda^{-\frac{\alpha+2}{\alpha}}\right)+c_2O\left(\lambda^{-\frac{\alpha+2}{\alpha}}\right)+c_1O\left(\lambda^{-1}\right)+c_2O\left(\lambda^{-\frac{3}{2}}\right).
	\end{split}
\end{align}

By the same arguments, but for the solutions $f_-(-x,\lambda)$ and $y(-x,\lambda)$, we obtain
\begin{align}\label{eq_minus}
	\begin{split}
&\sin\left(Q+\frac{\pi}{4}+\sqrt{\mu}b\right)\Big(d_1\left(-c_2+k_1^-\right)+d_2\left(c_1\sqrt{\mu}-k_2^-\right)\Big) \\&
	+\cos\left(Q+\frac{\pi}{4}+\sqrt{\mu}b\right)\Big(d_1\left(c_1\sqrt{\mu}-k_2^-\right)-d_2\left(-c_2+k_1^-\right)\Big) \\&
	=c_1\sqrt{\mu}O\left(\lambda^{-\frac{\alpha+2}{\alpha}}\right)+c_2O\left(\lambda^{-\frac{\alpha+2}{\alpha}}\right)+c_1O\left(\lambda^{-1}\right)+c_2O\left(\lambda^{-\frac{3}{2}}\right).
	\end{split}
\end{align}
Now we have two equations \eqref{eq_plus} and \eqref{eq_minus} and variables $\lambda$, $c_1=c_1(\lambda)$ and $c_2=c_2(\lambda)$. Without lost of generality, we assume that the eigenfunction, corresponding to $\lambda$, is normalized in the sence that
\begin{equation*}
	c_1 = \cos \phi(\lambda)
	\qquad c_2 = \sqrt{\mu} \sin\phi(\lambda)
\end{equation*}
for some $\phi(\lambda)\in [0,2\pi)$. Subtracting \eqref{eq_minus} from \eqref{eq_plus} and adding \eqref{eq_minus} to \eqref{eq_plus} give
\begin{align}\label{eq_for_D}
	\begin{split}
		&\sin\left(Q+\frac{\pi}{4}+\sqrt{\mu}b\right)\Big(2d_1c_2+d_1\left(k_1^+-k_1^-\right)-d_2\left(k_2^+-k_2^-\right)\Big)\\
		&-\cos\left(Q+\frac{\pi}{4}+\sqrt{\mu}b\right)\Big(2d_2c_2+d_1\left(k_2^+-k_2^-\right)+d_2\left(k_1^+-k_1^-\right)\Big)\\
		&=O\left(\lambda^{-\frac{\alpha+4}{2\alpha}}\right)+O\left(\lambda^{-1}\right)
	\end{split}
\end{align}
and
\begin{align}\label{eq_for_N}
	\begin{split}
		&\sin\left(Q+\frac{\pi}{4}+\sqrt{\mu}b\right)
		\Big(2d_2\sqrt{\mu}c_1+d_1\left(k_1^++k_1^-\right)-d_2\left(k_2^++k_2^-\right)\Big)\\
		&+\cos\left(Q+\frac{\pi}{4}+\sqrt{\mu}b\right)
		\Big(2d_1\sqrt{\mu}c_1-d_1\left(k_2^++k_2^-\right)-d_2\left(k_1^++k_1^-\right)\Big)\\
		&=O\left(\lambda^{-\frac{\alpha+4}{2\alpha}}\right)+O\left(\lambda^{-1}\right).
	\end{split}
\end{align}

We will distinguish two types of solutions, $\Lambda^D$ and $\Lambda^N$, of the system of equations \eqref{eq_for_D}-\eqref{eq_for_N}: we say that solution $\lambda\in \Lambda^D$ if $|\sin\phi(\lambda)|\geq|\cos\phi(\lambda)|$ and $\lambda\in \Lambda^N$ otherwise. Let $\{\hat{\lambda}_k\}_{k=1}^{\infty}$ and $\{\check{\lambda}_k\}_{k=1}^{\infty}$ be the increasing sequences of solutions of types $\Lambda^D$ and $\Lambda^N$,  respectively. We consider two cases: 1) $\lambda\in \Lambda^D$; 2) $\lambda\in \Lambda^N$.

\textbf{Case 1.} Assume $\lambda\in \Lambda^D$, so that $|\sin\phi(\lambda)|\geq 1/\sqrt{2}$. Therefore \eqref{eq_for_D}, together with Lemma \ref{k_pus_k} and \eqref{asympt_for_d}, implies
\begin{equation*}
\sin\left(Q+\frac{\pi}{4}+\sqrt{\mu}b\right)\sqrt{\mu}\sin\phi(\lambda)+O(1)=O(\lambda^{-\frac{\alpha+4}{2\alpha}})+O(\lambda^{-1}).
\end{equation*}
Since $|\sin\phi(\lambda)|\geq 1/\sqrt{2}$,
\begin{equation}\label{sin}
\sin\left(Q+\frac{\pi}{4}+\sqrt{\mu}b\right)=O\left(\lambda^{-\frac{1}{2}}\right).
\end{equation}
Therefore, solutions $\{\hat{\lambda}_k\}_{k=1}^{\infty}$ of type $\Lambda^D$ satisfy
\begin{equation}\label{eq_distfunc_D}
Q(b,\hat{\lambda}_n)+\frac{\pi}{4}+b\sqrt{\hat{\lambda}_n-q(b)}=n\pi+\beta(\hat{\lambda}_n),
\end{equation}
for some function $\beta(\lambda)$ tending to zero as $\lambda\rightarrow\infty$. To investigate \eqref{eq_distfunc_D} further, we need a better estimate on $\beta(\lambda)$. Combining \eqref{eq_for_D} and \eqref{eq_distfunc_D}  we obtain, for $\lambda=\hat{\lambda}_n$,
\begin{align}\label{key}
\begin{split}
&\sin\beta(\lambda)\Big(2d_1\sqrt{\mu}\sin\phi(\lambda)+d_1\left(k_1^+-k_1^-\right)-d_2\left(k_2^+-k_2^-\right)\Big)\\
&-\cos\beta(\lambda)\Big(2d_2\sqrt{\mu}\sin\phi(\lambda)+d_1\left(k_2^+-k_2^-\right)+d_2\left(k_1^+-k_1^-\right)\Big)\\
&=O\left(\lambda^{-\frac{\alpha+4}{2\alpha}}\right)+O\left(\lambda^{-1}\right).
\end{split}
\end{align}
Then, for $\lambda=\hat{\lambda}_n$, Lemma \ref{k_pus_k} and \eqref{asympt_for_d} imply that
\begin{align}\label{sinbcosb}
\begin{split}
&\sin\beta(\lambda)\Big(2\sqrt{\mu}\sin\phi(\lambda)+k_1^+-k_1^-+O(\lambda^{-\frac{1}{\alpha}})\Big)\\
&-\cos\beta(\lambda)\left(k_2^+-k_2^- +2d_2\sqrt{\mu}\sin\phi(\lambda)\right)= O\left(\lambda^{-1}\right)+O\left(\lambda^{-\frac{\alpha+2}{2\alpha}}\right).\\
\end{split}
\end{align}
Now, let us investigate \eqref{eq_for_N}. By using \eqref{asympt_for_d}, \eqref{sin} and Lemma \ref{k_pus_k}, we conclude, from \eqref{eq_for_N}, that $\sqrt{\mu}c_1=O(1)$, so that
\begin{equation}\label{c1}
c_1=\cos\phi(\lambda)=O(\lambda^{-\frac{1}{2}}),
\end{equation}
and consequently $k_1^{\pm}=O(\lambda^{-\frac{1}{2}})$. Therefore \eqref{sinbcosb} gives
\begin{align}
\begin{split}
\sin\beta(\lambda)=&\frac{\cos\beta(\lambda)}{2\sqrt{\mu}\sin\phi(\lambda)+O(\lambda^{-\frac{1}{2}})+O(\lambda^{-\frac{1}{\alpha}})}\Big(k_2^+-k_2^-+2d_2\sqrt{\mu}\sin\phi(\lambda)\Big)\\
&+O\left(\lambda^{-\frac{2\alpha+2}{2\alpha}}\right)+O\left(\lambda^{-\frac{3}{2}}\right).
\end{split}
\end{align}
and hence
\begin{align}\label{beta}
\begin{split}
\beta(\lambda)=\frac{1}{2\sqrt{\mu}\sin\phi(\lambda)}(k_2^+-k_2^-)+d_2+O\left(\lambda^{-\frac{2\alpha+2}{2\alpha}}\right)+O\left(\lambda^{-\frac{3}{2}}\right).
\end{split}
\end{align}
Thus, from \eqref{K2minusK2}, it follows
\begin{align}\label{*beta}
\begin{split}
	\beta(\lambda)=&\frac{1}{4\sqrt{\mu}}\int_{-b}^{b}(q(s)-q(b))ds+\frac{1}{4\sqrt{\mu}}\int_{-b}^{b}(q(s)-q(b))\cos(2\sqrt{\mu}s)ds\\
	&+\frac{\cos\phi(\lambda)}{4\sqrt{\mu}\sin\phi(\lambda)}\int_{-b}^{b}(q(s)-q(b))\sin(2\sqrt{\mu}s)ds+d_2+O\left(\lambda^{-1}\right),
\end{split}
\end{align}
so that the Riemann-Lebesgue lemma and \eqref{c1} imply
\begin{align*}
	\begin{split}
		\beta(\lambda)=&\frac{1}{4\sqrt{\mu}}\int_{-b}^{b}(q(s)-q(b))ds+\frac{1}{4\sqrt{\mu}}\int_{-b}^{b}V(s)\cos(2\sqrt{\mu}s)ds+d_2+O\left(\lambda^{-1}\right)
	\end{split}
\end{align*}
and consequently \eqref{eq_distfunc_D} gives
\begin{align}\label{lambda_hat}
\begin{split}
	n\pi&=Q(b,\hat{\lambda}_n)+\frac{\pi}{4}+b\sqrt{\hat{\lambda}_n-q(b)}-\frac{1}{4\sqrt{\hat{\lambda}_n}}\int_{-b}^{b}(q(s)-q(b))ds\\
	&-\frac{1}{4\sqrt{\hat{\lambda}_n}}\int_{-b}^{b}V(s)\cos(2\sqrt{\hat{\lambda}_n}s)ds-d_2(\hat{\lambda}_n)+O\left(\hat{\lambda}_n^{-1}\right).
	\end{split}
\end{align}

\textbf{Case 2.} Assume $\lambda\in \Lambda^N$, so that $|\cos\phi(\lambda)|\geq 1/\sqrt{2}$. Therefore \eqref{eq_for_N}, together with Lemma \ref{k_pus_k} and \eqref{asympt_for_d}, implies 
\begin{equation}\label{cos}
	\cos\left(Q+\frac{\pi}{4}+\sqrt{\mu}b\right)=O\left(\lambda^{-\frac{1}{2}}\right).
\end{equation}
Therefore, the solutions $\{\check{\lambda}_k\}_{k=1}^{\infty}$ of type $\Lambda^N$ satisfy
\begin{equation}\label{eq_distfunc_N}
Q(b,\check{\lambda}_n)+\frac{\pi}{4}+b\sqrt{\check{\lambda}_n-q(b)}=n\pi-\frac{\pi}{2}+\gamma(\check{\lambda}_n)
\end{equation}
By combining \eqref{eq_for_N}, \eqref{eq_distfunc_N} and then using Lemma \ref{k_pus_k} together with \eqref{asympt_for_d}, for $\lambda=\check{\lambda}_n$, we obtain
\begin{align}\label{singcosg}
	\begin{split}
	&\cos\gamma(\lambda)\left(k_1^++k_1^-+2d_2\sqrt{\mu}\cos\phi(\lambda)\right)\\
	&-\sin\gamma(\lambda)\Big(2\sqrt{\mu}\cos\phi(\lambda)-k_2^++k_2^-+O(\lambda^{-\frac{1}{\alpha}})\Big)
	=O\left(\lambda^{-1}\right)+O\left(\lambda^{-\frac{\alpha+2}{2\alpha}}\right).
\end{split}
\end{align}
Let us investigate \eqref{eq_for_D}. By using \eqref{asympt_for_d}, \eqref{cos} and Lemma \ref{k_pus_k}, we conclude, from \eqref{eq_for_D}, that $c_2=O(1)$, so that 
\begin{equation}\label{c2}
	\sin\phi(\lambda)=O(\lambda^{-\frac{1}{2}}),
\end{equation}
and consequently $k_2^{\pm}=O(\lambda^{-\frac{1}{2}})$. Therefore \eqref{singcosg} gives
\begin{align}\label{gamma}
	\begin{split}
	\sin\gamma(\lambda)=&\frac{\cos\gamma(\lambda)}{2\sqrt{\mu}\cos\phi(\lambda)+O(\lambda^{-\frac{1}{2}})+O(\lambda^{-\frac{1}{\alpha}})}\left(k_1^++k_1^-++2d_2\sqrt{\mu}\cos\phi(\lambda)\right)\\
	&+O\left(\lambda^{-\frac{2\alpha+2}{2\alpha}}\right)+O\left(\lambda^{-\frac{3}{2}}\right)
	\end{split}
\end{align}
and hence
\begin{equation*}
	\gamma(\lambda)=\frac{1}{2\sqrt{\mu}\cos\phi(\lambda)}(k_1^++k_1^-)+d_2+O\left(\lambda^{-\frac{2\alpha+2}{2\alpha}}\right)+O\left(\lambda^{-\frac{3}{2}}\right).
\end{equation*}
Therefore, from \eqref{K1minusK1}, it follows
\begin{align}\label{*gamma}
	\begin{split}
		\gamma(\lambda)=&\frac{1}{4\sqrt{\mu}}\int_{-b}^{b}(q(s)-q(b))ds+\frac{1}{4\sqrt{\mu}}\int_{-b}^{b}(q(s)-q(b))\cos(2\sqrt{\mu}s)ds\\
		&+\frac{\sin\phi(\lambda)}{4\sqrt{\mu}\cos\phi(\lambda)}\int_{-b}^{b}(q(s)-q(b))\sin(2\sqrt{\mu}s)ds+d_2+O(\lambda^{-1}),
	\end{split}
\end{align}
so that the Riemann-Lebesgue lemma and \eqref{c2} imply
\begin{align*}
	\begin{split}
		\gamma(\lambda)=&\frac{1}{4\sqrt{\mu}}\int_{-b}^{b}(q(s)-q(b))ds+\frac{1}{4\sqrt{\mu}}\int_{-b}^{b}V(s)\cos(2\sqrt{\mu}s)ds+d_2+O\left(\lambda^{-1}\right).
	\end{split}
\end{align*}
This and \eqref{eq_distfunc_N} give
\begin{align}\label{lambda_check}
\begin{split}
	n\pi=&Q(b,\check{\lambda}_n)+\frac{3\pi}{4}+b\sqrt{\check{\lambda}_n-q(b)}-\frac{1}{4\sqrt{\check{\lambda}_n}}\int_{-b}^{b}(q(s)-q(b))ds\\
	&-\frac{1}{4\sqrt{\check{\lambda}_n}}\int_{-b}^{b}V(s)\cos(2\sqrt{\check{\lambda}_n}s)ds-d_2(\check{\lambda}_n)+O\left(\check{\lambda}_n^{-1}\right).
	\end{split}
\end{align}
	From  \eqref{lambda_hat} and \eqref{lambda_check}, we notice that solutions of types $\Lambda^D$ and $\lambda^N$ are interlacing at infinity, that is
	\begin{equation}\label{interlacing}
	\check{\lambda}_n\leq \hat{\lambda}_n\leq \check{\lambda}_{n+1}\leq\hat{\lambda}_{n+1}\leq \cdots
\end{equation}
for sufficiently large $n\in \mathbb{N}$. Denote $\nu_{2n-1}=\check{\lambda}_n$ and $\nu_{2n}=\hat{\lambda}_n$. Then $\{\nu_j\}_{j=1}^{\infty}$ are solutions to \eqref{eq_plus}, \eqref{eq_minus}, and $\nu_{n+1}>\nu_{n}$ for sufficiently large $n\in \mathbb{N}$. Moreover, the asymptotic formulas \eqref{lambda_hat} and \eqref{lambda_check} give
\begin{align}\label{nu}
\begin{split}
\frac{\pi}{4}(2n-1)=&Q(b,\nu_n)+b\sqrt{\nu_n-q(b)}-\frac{1}{4\sqrt{\nu_n}}\int_{-b}^{b}(q(s)-q(b))ds\\
&-\frac{1}{4\sqrt{\nu_n}}\int_{-b}^{b}V(s)\cos(2\sqrt{\nu_n}s)ds-d_2(\nu_n)+O\left(\nu_n^{-1}\right)
\end{split}
\end{align}
as $n\rightarrow\infty$. As we mentioned before, the eigenvalues of $H$ are solutions to \eqref{eq_plus}, \eqref{eq_minus}. Therefore, to prove the theorem, it is sufficient to show 
\begin{equation}\label{sol_and_eig}
N(\lambda,H)=\#\{ \nu_j \le \lambda\}+O(1),
\quad \lambda>0,
\end{equation}
where $\#$ means the cardinality of a set, and $N(\lambda,H)$ is the counting function of $H$, that is
\begin{equation*}
	N(\lambda, H) := \# \{ \lambda' \le \lambda, \quad \lambda' \textnormal{ is an eigenvalue of } H \}=\#\{ \lambda_j \le \lambda\}.
\end{equation*}
In order to prove \eqref{sol_and_eig}, let us consider the operator $H_0$ in $L^2(\mathbb{R})$, generated by the expression
\begin{equation*}
	-\frac{d^2}{dx^2}+q_0(x),
	\quad
	q_0(x) = \sum_{j=1}^N c_j |x|^{\alpha_j}.
\end{equation*}
Since the operator of multiplication by $V(x)$ is bounded and symmetric in $L^2(\mathbb{R})$, \cite[Theorem 4.10]{Kato} implies
\begin{equation}\label{assymptotic_for_distr_funct2}
N(\lambda, H)=N(\lambda, H_0)+O(1),
\quad
\lambda \rightarrow\infty.
\end{equation}
On the other hand, \cite[formula (7.7.4)]{T} gives
\begin{equation}\label{assymptotic_for_distr_funct3}
N(\lambda, H_0)=\frac{2}{\pi}\int_{0}^{a(\lambda)}\sqrt{\lambda-q_0(t)}dt+O(1),
\quad
\lambda \rightarrow\infty.
\end{equation}
Finally, since $q_0(x)=q(x)$ for $x>b$, we estimate
\begin{align*}
	& \int_{0}^{a(\lambda)}\sqrt{\lambda-q_0(t)}dt-Q(b,\lambda)-b\sqrt{\lambda-q(b)}\\
	&=\int_{0}^{a(\lambda)}\sqrt{\lambda-q_0(t)}dt-\int_{b}^{a(\lambda)}\sqrt{\lambda-q(t)}dt-b\sqrt{\lambda-q(b)}\\
	&=\int_{0}^{b}\sqrt{\lambda-q_0(t)}dt-b\sqrt{\lambda-q(b)}=\int_{0}^{b}\frac{-q_0(t)+q(b)}{\sqrt{\lambda-q_0(t)}+\sqrt{\lambda-q(b)}}dt=O(\lambda^{-\frac{1}{2}}).
\end{align*}
Therefore \eqref{assymptotic_for_distr_funct2} and \eqref{assymptotic_for_distr_funct3} imply
\begin{equation*}
	N(\lambda, H)=\frac{2}{\pi}Q(b,\lambda)+\frac{2}{\pi}b\sqrt{\lambda-q(b)}+O(1). 
\end{equation*}
This and \eqref{nu} give \eqref{sol_and_eig}, which consequently implies that eigenvalues $\{\lambda_n\}_{n=1}^{\infty}$ satisfy the asymptotic formula \eqref{nu}, and \eqref{asympt_for_d} completes the proof.
\end{proof}

Next we derive Theorem \ref{main_th} from Theorem \ref{gen_th}.

\begin{proof}
	According to \cite[formula (38)]{MA}, the function $Q(x,\lambda)$, corresponding to the potential $q(x)=|x|^{\alpha}+V(x)$, has the following asymptotic behavior, as $\lambda\rightarrow\infty$,
	\begin{equation}\label{asympt_for_Q}
	Q(b,\lambda) = \frac{\Gamma\left(\frac{3}{2}\right) \Gamma\left(\frac{1}{\alpha}\right)}{\alpha \Gamma\left(\frac{3}{2}+\frac{1}{\alpha}\right)} \lambda^{\frac{\alpha+2}{2\alpha}} - b \sqrt{\lambda} + \frac{1}{2 \sqrt{\lambda}} b^{\alpha+1}(\alpha+1)^{-1} + O\left(\lambda^{-\frac{3}{2}}\right)
	\end{equation}
	for fixed $b>0$ such that $V(x)=0$ for $|x|\geq b$. Inserting this into \eqref{nu} gives
\begin{align}\label{L}
\begin{split}
\frac{\pi}{4}(2n-1)=&\frac{\Gamma\left(\frac{3}{2}\right) \Gamma\left(\frac{1}{\alpha}\right)}{\alpha \Gamma\left(\frac{3}{2}+\frac{1}{\alpha}\right)} \lambda_n^{\frac{\alpha+2}{2\alpha}}-b\sqrt{\lambda_n}+\frac{1}{2\sqrt{\lambda_n}}\cdot\frac{b^{\alpha+1}}{\alpha+1}+b\sqrt{\lambda_n-q(b)}\\
&-\frac{1}{2\sqrt{\lambda_n}}\cdot\frac{b^{\alpha+1}}{\alpha+1}+\frac{b^{\alpha+1}}{2\sqrt{\lambda_n}}-\frac{1}{4\sqrt{\lambda_n}}\int_{-b}^{b}V(s)ds\\
&-\frac{1}{4\sqrt{\lambda_n}}\int_{-b}^{b}V(s)\cos(2\sqrt{\lambda_n}s)ds-d_2(\lambda_n)+O\left(\lambda_n^{-1}\right).
\end{split}
\end{align}
Note that
\begin{equation*}
	\begin{gathered}
		\sqrt{\lambda_n-q(b)}-\sqrt{\lambda_n}+\frac{b^{\alpha}}{2\sqrt{\lambda_n}}=\frac{-q(b)}{\sqrt{\lambda_n-q(b)}+\sqrt{\lambda_n}}+\frac{b^{\alpha}}{2\sqrt{\lambda_n}}\\
		=-b^{\alpha}\frac{2\sqrt{\lambda_n}-\sqrt{\lambda_n-q(b)}-\sqrt{\lambda_n}}{2\sqrt{\lambda_n}\left(\sqrt{\lambda_n-q(b)}+\sqrt{\lambda_n}\right)}=O\left(\lambda_n^{-\frac{3}{2}}\right)
	\end{gathered}
\end{equation*}
and, see \cite[(49)]{MA},
\begin{equation*}
	d_2(\lambda)=
	\begin{cases}
		O(\lambda^{-1}) & \text{if }\alpha\leq 2,\\
		\frac{\alpha(\alpha-1)\Gamma\left(\frac{3}{2}+\frac{1}{\alpha}\right)\cot\frac{\pi}{\alpha}}{48(2+\alpha)\Gamma\left(\frac{3}{2}\right)\Gamma\left(\frac{1}{\alpha}\right)}\lambda^{-\frac{\alpha+2}{2\alpha}} & \text{if } \alpha>2.
	\end{cases}
\end{equation*}
Therefore we can rewrite \eqref{L} in the following way
\begin{align}\label{asymptotic_distr}
	\begin{split}
		\frac{\pi}{4}(2n-1)=&\frac{\Gamma\left(\frac{3}{2}\right) \Gamma\left(\frac{1}{\alpha}\right)}{\alpha \Gamma\left(\frac{3}{2}+\frac{1}{\alpha}\right)} \lambda_n^{\frac{\alpha+2}{2\alpha}}-\frac{1}{4\sqrt{\lambda_n}}\int_{-\infty}^{+\infty}V(s)ds\\
		&-\frac{\alpha(\alpha-1)\Gamma\left(\frac{3}{2}+\frac{1}{\alpha}\right)\cot\frac{\pi}{\alpha}}{48(2+\alpha)\Gamma\left(\frac{3}{2}\right)\Gamma\left(\frac{1}{\alpha}\right)}\lambda_n^{-\frac{\alpha+2}{2\alpha}}\\
		&-\frac{1}{4\sqrt{\lambda_n}}\int_{-b}^{b}V(s)\cos(2\sqrt{\lambda_n}s)ds+O(\lambda_n^{-1}).
\end{split}
\end{align}
as $n\rightarrow\infty$. This implies \eqref{main_asy}.
\end{proof}

\section{Remarks and consequences}
We start with the following remark.
\begin{remark}
	According to \cite[page 92]{SR}, the third and fourth terms are $O(n^{-\frac{\alpha\tau+2}{\alpha+2}})$ and $O(n^{-\frac{4}{\alpha+2}})$, respectively, so that they may change their order, depending on $V(x)$ and the relation between $\alpha$ and $\tau$. Moreover, in case $\alpha\leq 2$, the fourth term will be absorbed by $O(n^{-1})$. However, the third term, in general, is not absorbed by $O(n^{-1})$; see Example \ref{example1}.
\end{remark}

Next, we give an example demonstrating the effect of the parameter $ \tau $ on the eigenvalues.
\begin{example}\label{example1}
	Let $H$ be the self-adjoint operator in $L^2(\mathbb{R})$ generated by expression
	\begin{equation*}
		-\frac{d^2}{dx^2}+x^2+V(x),
	\end{equation*}
where, for $1>\tau>0$,
\begin{equation}\label{exampleV}
	V(x)=\begin{cases}
		\sum_{j=1}^{\infty}2^{-j\tau}\cos(2^jx) & \text{if } |x|\leq\pi,\\
		0& \text{otherwise }.
	\end{cases}
\end{equation}
By \cite[Theorem 4.9]{Zygmund}, $V(x)$ is a H{\"o}lder continuous function with exponent $\tau>0$ in $[-\pi,\pi]$. Therefore Theorem \ref{main_th} shows that eigenvalues of $H$ satisfy
\begin{align*}
	\begin{split}
		\lambda_n=2n-1+\frac{1}{4\sqrt{2n}}\int_{-\pi}^{\pi}V(s)ds+\frac{1}{4\pi\sqrt{2n}}\int_{-\pi}^{\pi}V(s)\cos(2\sqrt{2n}s)ds+O(n^{-1}).
    \end{split}
\end{align*}
Let us consider the subsequence of the eigenvalues $\{\lambda_{n_k}\}_{k=2}^{\infty}$ with $n_k=2^{2k-3}$. Since
\begin{equation*}
	\int_{-\pi}^{\pi}V(s)\cos(2\sqrt{2n_k}s)ds
	=\int_{-\pi}^{\pi}\sum_{j=1}^{\infty}2^{-j\tau}\cos(2^js)\cos(2^ks)ds=\pi2^{-k\tau},
\end{equation*}
we derive the asymptotic for the subsecuence of the eigenvalues of $H$, $\{\lambda_{n_k}\}_{k=2}^{\infty}$,
\begin{equation*}
	\lambda_{n_k}=2n_k-1+n_k^{-\frac{1}{2}}\frac{1}{4\sqrt{2}}\int_{-\pi}^{\pi}V(s)ds +n_k^{-\frac{1+\tau}{2}}2^{-\frac{5+3\tau}{2}}+O(n_k^{-1}).
\end{equation*}
\end{example}

Next, we give two examples for which the asymptotic formulas \eqref{asy_gen} can be written more explicitly.
\begin{example}
	Assume that $V(x)$ satisfies conditions of Theorem \ref{gen_th} and $\alpha\in \mathbb{N}$, $c\in \mathbb{R}$. Let $H$ be the self-adjoint operator in $L^2(\mathbb{R})$ generated by the expression
	\begin{equation*}
		-\frac{d^2}{dx^2} + (|x|+c)^{\alpha}+V(x).
	\end{equation*}
	By Theorem \ref{gen_th}, the eigenvalues of $H$ satisfy \eqref{asy_gen}. By \cite[(38)]{MA}, we compute:
	\begin{align*}
		\begin{split}
			Q(b,\lambda)&=\int_{b}^{a(\lambda)}(\lambda-(t+c)^{\alpha})^{\frac{1}{2}}dt=\lambda^{\frac{\alpha+2}{2\alpha}}\int_{\frac{b+c}{\lambda^{1/\alpha}}}^{1}(1-t^\alpha)^{\frac{1}{2}}dt\\
			&=\frac{\Gamma\left(\frac{3}{2}\right) \Gamma\left(\frac{1}{\alpha}\right)}{\alpha \Gamma\left(\frac{3}{2}+\frac{1}{\alpha}\right)} \lambda^{\frac{\alpha+2}{2\alpha}} - (b+c) \sqrt{\lambda} + \frac{1}{2 \sqrt{\lambda}} (b+c)^{\alpha+1}(\alpha+1)^{-1} + O\left(\lambda^{-\frac{3}{2}}\right)
		\end{split}
	\end{align*}
	Therefore \eqref{asy_gen} gives
	\begin{align*}
		\begin{split}
			\frac{\pi}{4}(2n-1)&=\frac{\Gamma\left(\frac{3}{2}\right) \Gamma\left(\frac{1}{\alpha}\right)}{\alpha \Gamma\left(\frac{3}{2}+\frac{1}{\alpha}\right)} \lambda_n^{\frac{\alpha+2}{2\alpha}}
			-c\sqrt{\lambda_n}-\frac{1}{4\sqrt{\lambda_n}}\left(\frac{2c^{\alpha+1}}{\alpha+1}-\int_{-\infty}^{+\infty}V(s)ds\right)\\
			&-\frac{1}{4\sqrt{\lambda_n}}\int_{-b}^{b}V(s)\cos(2\sqrt{\lambda_n}s)ds+O\left(\lambda_n^{-\frac{\alpha+2}{2\alpha}}\right)+O\left(\lambda_n^{-1}\right).\\
		\end{split}
	\end{align*}
	This gives the first two terms of the heat trace. Indeed, this implies that
	\begin{equation*}
		N(\lambda)=\frac{1}{2}C_1 \lambda^{\frac{\alpha+2}{2\alpha}}-\frac{2c}{\pi}\lambda^{\frac{1}{2}}+O(1)
	\end{equation*}
	By applying the Watson's Lemma \cite[Lemma 2.2]{PS}, we obtain
	\begin{equation*}
		\int_{0}^{\infty}e^{-t\lambda}N(\lambda)d\lambda=
		\frac{1}{2}C_1 \Gamma\left(1+\frac{\alpha+2}{2\alpha}\right) t^{-\frac{\alpha+2}{2\alpha}-1}-\frac{c}{\sqrt{\pi}}t^{-\frac{3}{2}}+O(t^{-1})
	\end{equation*}
	as $t\rightarrow +0$. Therefore, as $t\rightarrow+0$,
	\begin{align*}
		\begin{split}
			\sum_{n=1}^{\infty}e^{-t\lambda_n}=\int_{0}^{\infty}e^{-t\lambda}dN(\lambda)&=t\int_{0}^{\infty}e^{-t\lambda}N(\lambda)d\lambda\\
			&=\frac{1}{\sqrt{\pi}}\Gamma\left(\frac{\alpha+1}{\alpha}\right)t^{-\frac{\alpha+2}{2\alpha}}-\frac{c}{\sqrt{\pi}}t^{-\frac{1}{2}}+O\left(1\right).
		\end{split}
	\end{align*}
\end{example}
    By the same way, one can derive the first term of the heat trace corresponding to the operator defined in Theorem \ref{main_th} and verify that the perturbation does not affect the first term:
    \begin{cor}
    	Under the conditions of Theorem \ref{main_th}, the heat trace satisfies
    	\begin{equation*}
    		\sum_{n=1}^{\infty}e^{-t\lambda_n}=\frac{1}{\sqrt{\pi}}\Gamma\left(\frac{\alpha+1}{\alpha}\right)t^{-\frac{\alpha+2}{2\alpha}}+O\left(1\right),
    		\quad \text{as }
    		\quad t\rightarrow +0.
    	\end{equation*}
    \end{cor}

The next example is a quartic AHO, for which $Q(b,\lambda)$ expands in terms of $\{\lambda^{-\frac{k}{4}}\}_{k=-3}^{+\infty}$. Consequently, the asymptotic formula can be written more explicitly.

\begin{example}[Quartic AHO]
	Assume that function $V(x)$ satisfies conditions of Theorem \ref{gen_th} and $c\in\mathbb{R}$. Let $H$ be the self-adjoint operator in $L^2(\mathbb{R})$ generated by
	\begin{equation*}
		-\frac{d^2}{dx^2}+(x^2+c)^2+V(x), \qquad c\in \mathbb{R}.
	\end{equation*}
Then, by Theorem \ref{gen_th}, the eigenvalues of $H$ satisfy \eqref{asy_gen}. Let us investigate $Q(b,\lambda)$:
\begin{align*}
	Q(b,\lambda)&=\int_{b}^{a(\lambda)}[\lambda-(x^2+c)^2]^{\frac{1}{2}}dx=\int_{b}^{(\lambda^{\frac{1}{2}}-c)^{\frac{1}{2}}}[\lambda-(x^2+c)^2]^{\frac{1}{2}}dx\\
	&=\lambda^{\frac{3}{4}}\int_{\frac{b}{(\lambda^{\frac{1}{2}}-c)^{\frac{1}{2}}}}^{1}\left[1-\left(\left(1-\frac{c}{\lambda^{\frac{1}{2}}}\right)y^2+\frac{c}{\lambda^{\frac{1}{2}}}\right)^2\right]^{\frac{1}{2}}
	\left(1-\frac{c}{\lambda^{\frac{1}{2}}}\right)^{\frac{1}{2}}dy.
\end{align*}
Let $\lambda=1/r^4$ and 
\begin{equation*}
	g(r):=\int_{\frac{br}{(1-cr^2)^{\frac{1}{2}}}}^{1}\left[1-((1-cr^2)y^2+cr^2)^2\right]^{\frac{1}{2}}(1-cr^2)^{\frac{1}{2}}dy.
\end{equation*}
By computing the Taylor expansion of $g(r)$ at $r=0$, one can express $Q(b,\lambda)$ in terms of $\{\lambda^{-\frac{n}{4}}\}$. Therefore, by \eqref{asy_gen}, we conclude
\begin{equation*}
	\frac{\pi}{4}(2n-1)=\sum_{k=0}^{6}a_k\lambda^{\frac{3-k}{4}}-\frac{1}{4\sqrt{\lambda_n}}\int_{-b}^{b}V(s)\cos(2\sqrt{\lambda_n}s)ds+O(\lambda^{-1}),
\end{equation*}
were $a_k=g^{(k)}(0)/k!$ for $k\neq 1$, $5$ and
\begin{equation*}
	a_1=g^{(1)}(0)-1, \quad a_5=g^{(5)}(0)/5!-\frac{1}{4}\int_{-b}^{b}q(s)ds.
\end{equation*}
\end{example}

\begin{remark}\label{Remark}
	Similarly, one can investigate AHO with potential $q(x)=(|x|^{\alpha}+c)^{n}$ with $\alpha>0$, $m\in\mathbb{N}$, and derive more explicit forms of \eqref{asy_gen}.
\end{remark}

We end this section by considering the operators in $L_2[0,\infty)$ generated by the expression \eqref{def_patential1} with Dirichlet and Neumann boundary conditions, respectively. 
\begin{cor}\label{asy_for_DN}
	Assume that $V(x)$ satisfies the conditions of Theorem \ref{main_th}. Let $\{\lambda^D_j\}_{j=1}^{\infty}$ and $\{\lambda^N_j\}_{j=1}^{\infty}$ be the eigenvalues of the operators in $L_2[0,\infty)$ generated by \eqref{def_patential1} with Dirichlet and Neumann boundary conditions, respectively. Then
	\begin{align}\label{asy_for_D}
		\begin{split}
			\lambda_n^N&=C_1^{-\frac{2\alpha}{\alpha+2}}(4n-1)^{\frac{2\alpha}{\alpha+2}}+\frac{2\alpha}{\alpha+2}C_0C_1^{-\frac{\alpha+1}{\alpha+2}}(4n-3)^{-\frac{2}{\alpha+2}}\\
			&+\frac{2\alpha}{\alpha+2}\frac{1}{4\pi}C_1^{-\frac{\alpha+4}{\alpha+2}}(4n-1)^{-\frac{2}{\alpha+2}}\int_{0}^{\infty}V(s)\cos\left(2C_1^{-\frac{\alpha}{\alpha+2}}(4n-1)^{\frac{\alpha}{\alpha+2}}s\right)ds\\
			&+\frac{2\alpha}{\alpha+2}C_2C_1^{-\frac{\alpha+6}{\alpha+2}}(4n-1)^{-\frac{4}{\alpha+2}}+O\left(n^{-1}\right),
		\end{split}
	\end{align}
	\begin{align}\label{asy_for_N}
		\begin{split}
			\lambda_n^N&=C_1^{-\frac{2\alpha}{\alpha+2}}(4n-3)^{\frac{2\alpha}{\alpha+2}}+\frac{2\alpha}{\alpha+2}C_0C_1^{-\frac{\alpha+4}{\alpha+2}}(4n-3)^{-\frac{2}{\alpha+2}}\\
			&+\frac{2\alpha}{\alpha+2}\frac{1}{4\pi}C_1^{-\frac{\alpha+4}{\alpha+2}}(4n-3)^{-\frac{2}{\alpha+2}}\int_{0}^{\infty}V(s)\cos\left(2C_1^{-\frac{\alpha}{\alpha+2}}(4n-3)^{\frac{\alpha}{\alpha+2}}s\right)ds\\
			&+\frac{2\alpha}{\alpha+2}C_2C_1^{-\frac{\alpha+6}{\alpha+2}}(4n-3)^{-\frac{4}{\alpha+2}}+O\left(n^{-1}\right),
		\end{split}
	\end{align}
	where $C_0=\frac{1}{\pi}\int_{0}^{\infty}V(x)dx$ and $C_1$, $C_2$ are the constants defined in Theorem \ref{main_th}.
\end{cor}

\begin{proof}
	Consider the solution of \eqref{diff_eq} in $[0,b]$ with boundary conditions $y(0)=1$, $y'(0)=0$, and the solution of \eqref{diff_eq} in $[b,\infty)$. By gluing them together, we obtain \eqref{eq_plus} (with $c_1=1$ and $c_2=0$)
	\begin{align}
		\begin{split}
			\sin&\left(Q+\frac{\pi}{4}+\sqrt{\mu}b\right)\Big(d_1k_1^++d_2\sqrt{\mu}-d_2k_2^+\Big)\\
			&+\cos\left(Q+\frac{\pi}{4}+\sqrt{\mu}b\right)\Big(d_1\sqrt{\mu}-d_1k_2^+-d_2k_1^+\Big)=O\left(\lambda^{-\frac{\alpha+4}{2\alpha}}\right)+O\left(\lambda^{-1}\right).\\
		\end{split}
	\end{align}
	The same analysis we did for \eqref{eq_for_N}(in order to investigate $\check{\lambda}_n$) yields formula \eqref{lambda_check}, which consequently implies \eqref{asy_for_N}. Similarly, one can prove \eqref{asy_for_D}.
\end{proof}

\begin{remark}
	Corollary \ref{asy_for_DN} for Dirichlet boundary condition is obtained in \cite{MA}.
\end{remark}

\bibliography{fn}
\bibliographystyle{amsalpha}

\end{document}

%% file: praeambel.tex
\usepackage{ulem}

\newcommand{\apref}[3]{\hyperref[#2]{#1\ref*{#2}#3}}

\usepackage{enumerate}
\usepackage[latin1]{inputenc}
\usepackage{dsfont}
\usepackage{amssymb,amsthm,amsmath}

\input{xy}
\xyoption{all}

\usepackage{mathrsfs}

\theoremstyle{plain}
\newtheorem{prop}{Proposition}[section]
\newtheorem{lemma}[prop]{Lemma}

\newtheorem{thm}[prop]{Theorem}

\newtheorem{cor}[prop]{Corollary}

\theoremstyle{definition}

\newtheorem{example}[prop]{Example}

\theoremstyle{remark}
\newtheorem{remark}[prop]{Remark}

%% file: makros.tex
\newcommand\N{\mathbb{N}}

\newcommand\R{\mathbb{R}}

